\documentclass{article}

\usepackage{amsmath,amsthm,amssymb,enumerate,epsfig,color}

\def\<{\langle}
\def\>{\rangle}

\renewcommand{\phi}{\varphi}
\newtheorem{theorem}{Theorem}[section]
\newtheorem{lemma}[theorem]{Lemma}
\newtheorem{proposition}[theorem]{Proposition}
\newtheorem{corollary}[theorem]{Corollary}

\newtheorem{definition}[theorem]{Definition}

\title{On reduction curves and Garside properties of braids}
\author{Juan Gonz\'alez-Meneses}
\date{April 21, 2010}

\begin{document}

\maketitle

%|<------------------------------------------------------------------------>|

\begin{abstract}
In this paper we study the reduction curves of a braid, and how they can be used to decompose the braid into simpler ones in a precise way, which does not correspond exactly to the decomposition given by Thurston theory. Then we study how a cyclic sliding (which is a particular kind of conjugation) affects the normal form of a braid with respect to the normal forms of its components. Finally, using the above methods, we provide the example of a family of braids whose sets of sliding circuits (hence ultra summit sets) have exponential size with respect to the number of strands and also with respect to the canonical length.
\end{abstract}

\section{Introduction}

Braids can be seen as isotopy classes of orientation-preserving automorphisms of the $n$-times punctured disc $D_n$, that is, the braid group on $n$ strands $B_n$ is isomorphic to the mapping class group $\mathcal M(D_n)$. Attending to the Nielsen-Thurston classification of mapping classes, braids can be periodic, reducible or pseudo-Anosov. In this paper we shall study reducible braids, which are those braids preserving a family of disjoint, non-degenerate, simple closed curves in $D_n$.

From the algebraic point of view, the braid group $B_n$ has a well known lattice structure. The submonoid $B_n^+\subset B_n$ consists of those elements of $B_n$ which can be written as positive powers of the standard generators $\sigma_1,\ldots,\sigma_{n-1}$. This monoid defines a partial order of $B_n$ given by $a\preccurlyeq b \ \Leftrightarrow \ a^{-1}b\in B_n^+$. This is a lattice order, which is invariant under left multiplication. The triple $(B_n,B_n^+,\Delta)$, where $\Delta=\sigma_1(\sigma_2\sigma_1)\cdots(\sigma_{n-1}\sigma_{n-2}\cdots \sigma_1)$ is the braid known as {\it half twist} or {\it Garside element}, determines a Garside structure of the braid group $B_n$~\cite{DP}. This structure, first discovered by Garside~\cite{Garside}, has been very useful for showing many properties of $B_n$, as well as for providing a substantial number of solutions to the word problem and the conjugacy problem in this group. These algorithms can also be used in other groups sharing the same algebraic properties, which are known under the common name of {\it Garside groups}~\cite{DP}.

One of the latest solutions to the conjugacy problem in Garside groups, thus in braid groups, is given in~\cite{GG1} (see also~\cite{GG2}). In that paper, the {\it cyclic sliding} is defined as a special conjugation that can be applied to any given braid. Iterated application of cyclic sliding conjugates any braid to another one which has minimal length in its conjugacy class, and has some other good algebraic properties~\cite{GG1}. But in the case of braid groups, reducible braids can behave not so nicely with respect to cyclic sliding, as we shall see. Hence, if one is interested in conjugacy properties of braids, a deeper study of the relation between the geometric and algebraic properties of braids is needed.

It is well known that the essential reduction curves of a given braid decompose it into `pieces', or components, in the spirit of Thurston's decomposition of a mapping class. If the decomposition is done in the appropriate way, each component is again a braid, with fewer number of strands. These components are, in general, not carefully defined in the papers dealing with reducible braids: A reference to Thurston's decomposition of a mapping class is given instead. But the decomposition in the case of braids is not exactly the same, at least if one needs each of the resulting components to be a braid. In Section~\ref{S:reducible} we will explain the notion of reducible braid, and the decomposition of such a braid into braid components.

In Section~\ref{S:Garside} we briefly recall the notions we need from Garside theory. In particular we will define cyclic sliding and the {\it set of sliding circuits} of a braid. These sets are the ones computed in~\cite{GG1} to solve the conjugacy problem in Garside groups.

The normal form of a braid obtained from the Garside structure of $B_n$ is related to the normal form of each of its components, and this relation is more clear in the case in which the reduction curves are isotopic to geometric circles. We will study this in Section~\ref{S:CS and curves}, and we will see how the application of a cyclic sliding may transform the normal form of a reducible braid, and of each of its components.

Finally, using the results from previous sections we will provide, in Section~\ref{S:big_sets}, an example of a family of braids whose sets of sliding circuits have exponential size, with respect to the number of strands and also with respect to the length of the braids. We believe this is the first example of this kind.

\section{Reducible braids}\label{S:reducible}

\subsection{Reduction curves}

Let $D_n$ be the $n$-times punctured closed disc $D^2\backslash\{P_1,\ldots,P_n\}$. For simplicity, we will assume that $D^2$ is embedded in the complex plane $\mathbb C$, that its boundary is a geometric circle, and the $n$ punctures are the first $n$ natural numbers. A simple closed curve $\mathcal C$ in $D_n$ is said to be non-degenerate if it is not isotopic to a puncture or to the boundary of $D_n$, that is, if it encloses more than one and less than $n$ punctures. We will consider such curves up to isotopy, so we will denote by $[\mathcal C]$ the isotopy class of a curve $\mathcal C$. The term {\it curve} in this paper will be applied to either a particular non-degenerate, simple closed curve, or its isotopy class.

From now on, a {\it family of curves} $\mathcal F$ will mean a family of disjoint, non-degenerate, simple closed curves in $D_n$. Its isotopy class will be denoted $[\mathcal F]$. We will say that a curve $\mathcal C$ is {\it round} if it is isotopic in $D_n$ to a geometric circle. A family of curves $\mathcal F$ is round if each of its curves is round, or equivalently, if $\mathcal F$ is isotopic to a family of geometric circles.

A braid, being an isotopy class of automorphisms of $D_n$, acts on the set of (isotopy classes of) simple closed curves in $D_n$. Given a curve $\mathcal C$ and a braid $\beta\in B_n$, we will denote by $[\mathcal C]^\beta$ the curve obtained from $[\mathcal C]$ after the action induced by $\beta$. Similarly, we use the notation $[\mathcal F]^\beta$, where $[\mathcal F]$ is a family of curves.

A braid $\beta\in B_n$ is said to be {\it reducible} if there exists a family of curves $[\mathcal F]$ such that $[\mathcal F]^\beta = [\mathcal F]$. Equivalently, $\beta$ is reducible if there exist a curve $\mathcal C$ and a positive integer $m$ such that $[\mathcal C]^{\beta^m}=[\mathcal C]$ and the curves $\{\mathcal C, \mathcal C^{\beta},\ldots,\mathcal C^{\beta^{m-1}}\}$ are pairwise disjoint.  Such curves are called {\it reduction curves} of $\beta$. For instance, the braid $\beta=(\sigma_1 \sigma_2 \sigma_3\sigma_4 \sigma_5)^2\in B_6$ is a reducible braid, as it preserves the family of round curves $\mathcal F=\{\mathcal C_{1,2}, \mathcal C_{3,4}, \mathcal C_{5,6}\}$, where $C_{i,j}$ is the round curve determined by a geometric circle enclosing punctures $i$ to $j$ (see Figure~\ref{F:trenza_1}). In this example $\beta^3$ preserves $\mathcal F$ curve-wise.

\begin{figure}[ht]
\centerline{\includegraphics{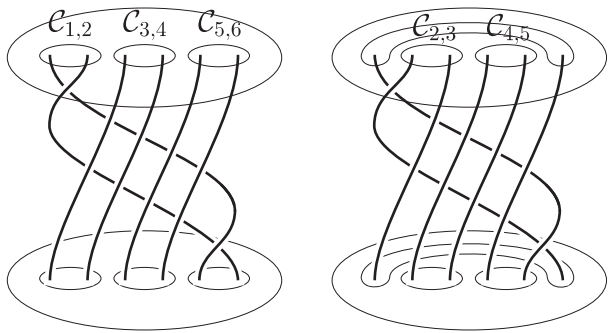}} \caption{The braid $(\sigma_1 \sigma_2 \sigma_3\sigma_4 \sigma_5)^2$ preserves two families of reduction curves.}
\label{F:trenza_1}
\end{figure}

Reduction curves are useful for decomposing braids into simpler ones. But a given braid may admit many (even infinite) distinct families of reduction curves, so the mentioned decompositions are a priori not unique. Nevertheless, there is a special, uniquely defined family of reduction curves of a braid, formed by the so called {\it essential reduction curves}~\cite{BLM}. In order to define them, we recall that the geometric intersection number of two curves $\mathcal C$ and $\mathcal D$ in a manifold, denoted $i(\mathcal C,\mathcal D)$, is the smallest cardinal of a set in $\{\mathcal C'\cap \mathcal D';\ \mathcal C'\in [\mathcal C], \: \mathcal D'\in [\mathcal D]\}$. Given $\beta\in B_n$, a curve $[\mathcal C]$ is said to be an essential reduction curve of $\beta$ if:
\begin{enumerate}

\item $[\mathcal C]$ is a reduction curve of $\beta$.

\item If $[\mathcal D]$ is such that $i(\mathcal C,\mathcal D)>0$, then $[\mathcal D]^{\beta^m}\neq [\mathcal D]$ for every integer $m>0$.

\end{enumerate}

The set of essential reduction curves of $\beta$ is called the {\it canonical reduction system} of $\beta$, and it is denoted $CRS(\beta)$. Notice that $CRS(\beta)$ is a family of disjoint reduction curves of $\beta$. Some good properties of $CRS(\beta)$ are that $CRS(\beta)=CRS(\beta^m)$ for every $m\neq 0$, and that $CRS(\alpha^{-1}\beta\alpha)=CRS(\beta)^\alpha$ for every $\alpha,\beta\in B_n$. Also, $CRS(\beta)=\emptyset$ if and only if $\beta$ is either periodic or pseudo-Anosov. For instance, in the example of Figure~\ref{F:trenza_1}, $CRS(\beta)=\emptyset$: The curves of $\mathcal F=\{\mathcal C_{1,2}, \mathcal C_{3,4}, \mathcal C_{5,6}\}$ are not essential, since they have nonzero geometric intersection with the round reduction curves $C_{2,3}$ and $C_{4,5}$. Actually, in that example $\beta$ is periodic, as $\beta^3=\Delta^2$ (recall that a braid is periodic if it has a nontrivial power belonging to the center of $B_n$, that is, to $\langle \Delta^2 \rangle$).

\subsection{Decomposition of a reducible braid}

Let $\beta$ be a non-periodic, reducible braid, so that $CRS(\beta)\neq \emptyset$. This canonical reduction system can be used to decompose $\beta$ into smaller braids, as we shall now see. The basic idea is to consider the action induced by $\beta$ on the connected components of $D_n\backslash\{CRS(\beta)\}$. Roughly speaking, each of these components is a punctured disc, so these restrictions are homeomorphisms of punctured discs, up to isotopy, and they can be considered as braids, which are usually called the {\it components} of $\beta$.  But there are several ambiguities in this definition, which the reader probably noticed. We will discuss about them in this section, so that we will be able to give a precise definition for the components of a braid $\beta\in B_n$.

Using the theory of mapping classes, Thurston's decomposition theorem (together with the definition of canonical reduction systems in~\cite{BLM}) states that $CRS(\beta)$ decomposes $D_n$ into two (not necessarily connected) invariant subsurfaces, such that $\beta$ restricted to one of them is pseudo-Anosov, and restricted to the second one has finite order (up to isotopy in a collar neighborhood of $CRS(\beta$)). We recall that braid groups have no torsion, but this does not cause any conflict with the above statement: Thurston's decomposition theorem deals with mapping classes in which the admissible isotopies fix the boundary setwise (and the theorem allows isotopy in a collar neighborhood of $CRS(\beta)$), while the admissible isotopies for braids fix the boundary of $D_n$ pointwise. For instance $\Delta^2$ is trivial when considered as a mapping class in Thurston's sense, as it corresponds to a Dehn twist along a curve parallel to the boundary $\partial(D_n)$, while it is certainly not a trivial braid. Hence $\Delta$ has finite order as a mapping class, but not as a braid.

In general, we do not want to decompose $\beta$ into a couple of mapping classes, as above, since in that case the resulting subsurfaces are not necessarily punctured discs, and the resulting mapping classes do not correspond to braids. We prefer to treat each connected component of $D_n\backslash CRS(\beta)$ independently.

Notice first that there is only one component of $D_n\backslash CRS(\beta)$ which is isomorphic to a punctured {\it closed} disc, namely the one containing the boundary of $D_n$. All other components are isomorphic to punctured open discs, that is, to punctured spheres. In order to avoid this situation, one can paste each curve of $CRS(\beta)$ to a connected component, in the following way.

Let $\mathcal F=CRS(\beta)\cup \{\partial(D_n)\}$. As $D_n$ embeds in the complex plane, and we are dealing with simple closed curves, we can rigourously talk about parts of $D_n$ {\it enclosed by} these curves. Then, for each curve $\mathcal C\in \mathcal F$, there is exactly one component $X_{\mathcal C}$ of $D_n\backslash \mathcal F$ which is enclosed by $\mathcal C$, and such that $\mathcal C\subset \overline {X_{\mathcal C}}$. We can then define $D_{\mathcal C}=X_{\mathcal C}\cup \mathcal C$, which is homeomorphic to a punctured closed disc with at least two punctures, and we can decompose $D_n$ as:
$$
    D_n=\bigsqcup_{\mathcal C\in \mathcal F}{D_{\mathcal C}}.
$$

Now notice that $\beta$ preserves $CRS(\beta)$ set-wise, but not necessarily curve-wise. For each $\mathcal C\in CRS(\beta)$, if $\beta$ sends $\mathcal C$ to $\mathcal C'\in CRS(\beta)$, then the punctured closed disc $D_{\mathcal C}$ is sent to $D_{\mathcal C'}$. Hence, the restriction of our braid $\beta$ to this component is a homeomorphism $\beta_{\mathcal C}: D_{\mathcal C} \rightarrow D_{\mathcal C'}$. Both $D_{\mathcal C}$ and $D_{\mathcal C'}$ are homeomorphic to closed discs with the same number of punctures, so $\beta_{\mathcal C}$ can be considered to be a braid, but not in a canonical way. Distinct homeomorphisms taking $D_{\mathcal C}$ and $D_{\mathcal C'}$ to the same punctured disc, yield distinct braids representing $\beta_{\mathcal C}$. Hence $\beta_{\mathcal C}$ is not well defined in this way, unless we are able to find a canonical way to send each $D_{\mathcal C}$ to a standard punctured disc.

Even if we find such a canonical way, there is another problem: Braids are defined up to isotopy fixing the boundary and the punctures, but one is allowed to {\it move} a curve in the interior of $D_n$. For instance, suppose that $[\mathcal C]\neq [\mathcal C']$. Now consider two curves $\mathcal C_1$ and $\mathcal C_2$ in a collar neighborhood of $\mathcal C$, both parallel to and disjoint from $\mathcal C$, one enclosing $\mathcal C$ and the other one enclosed by $\mathcal C$. Let $\tau_1$ and $\tau_2$ be Dehn twists along $\mathcal C_1$ and $\mathcal C_2$, respectively. Then the map $\gamma=\tau_2^{-1}\circ \tau_1\circ \beta$ is isotopic to $\beta$, so $\beta$ and $\gamma$ represent the same braid. But their restrictions to $D_{\mathcal C}$ do not coincide: They differ by a Dehn twist along the boundary, that is, one equals the other multiplied by the half twist $\Delta$.

In the above example, if $[\mathcal C]=[\mathcal C']$, that is, if $\beta$ sends $[\mathcal C]$ to itself, the restrictions of $\beta$ and $\gamma$ to $D_{\mathcal C}$ are conjugate by $\Delta$. Actually, if $CRS(\beta)$ is preserved by $\beta$ curve-wise, the restrictions of $\beta$ to each $D_{\mathcal C}$ are well defined up to conjugacy. But even in this case we are not happy enough.

An alternative approach could be to consider $\beta$ as a collection of strands: The three dimensional representation of a motion of $n$ punctures in the disc $D$. This motion can be extended to an isotopy of $D_n$ which sends the set of $n$ punctures to itself. Hence the curves in $CRS(\beta)$ can also be thought as {\it moving} curves, which trace a kind of {\it tube} enclosing some punctures. We could try to define the restriction of $\beta$ to $D_\mathcal C$ by considering only some strands starting inside the tube corresponding to $\mathcal C$. More precisely, consider $\overline{D_{\mathcal C}}\backslash D_{\mathcal C}$. This is a family of points and curves. The curves, $\mathcal C_1,\ldots,\mathcal C_r$, are the outermost curves in $CRS(\beta)$ enclosed by $\mathcal C$. The points correspond to the punctures of $D_n$ enclosed by $\mathcal C$ but not by $\mathcal C_1,\ldots,\mathcal C_r$. We could then try to define the braid $\beta_{\mathcal C}$ as the braid formed by the strands corresponding to those punctures, together with the {\it fat strands} corresponding to the tubes formed by $\mathcal C_1,\ldots,\mathcal C_r$. The problem here is that the curve $\mathcal C_i$ is not necessarily round, so its corresponding tube can be deformed in such a way that it is not clear to see how one could consider it as a strand. We can try to solve this problem in the following way.

Given a braid $\beta$, we can consider a {\it subbraid} by erasing some of its strands, that is, by filling some of the punctures of $D_n$. This is a priori not well defined, as the set of punctures that remain, say $I$, is not necessarily preserved by the braid. Actually, we would obtain a {\it partial braid}, an element of the fundamental groupoid of the configuration space of $m$ points in $D$ (\cite{Artin}, see also~\cite{EL}). There is a canonical element $\alpha_{I}$ in this groupoid, sending the points $\{1,\ldots,m\}$ to $I$, in which the punctures lie all the time along their motion in the diameter of $D$ corresponding to the real line (or in which the image of the diameter is the diameter). Then, if we denote by $\widetilde\beta_I$ the element of the groupoid obtained from $\beta$ by keeping only the strands starting at $I$, we can define $\beta_I = \alpha_I \widetilde\beta_I \alpha_{\beta(I)}^{-1}$, in which both the starting and ending points are $\{1,\ldots,m\}$, so $\beta_I$ is a well defined braid on $m$ strands, that we will call the {\it subbraid} of $\beta$ corresponding to $I$.  Notice that if one draws $\beta$ as a flat diagram, then $\beta_I\in B_m$ is the braid whose strands cross exactly in the same way as the strands starting at $I$ cross in $\beta$. See Figure~\ref{F:subtrenza}.

\begin{figure}[ht]
\centerline{\includegraphics{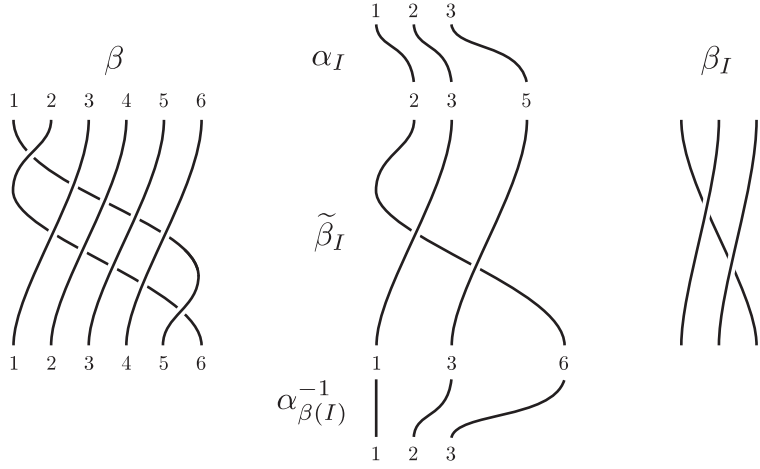}} \caption{The subbraid $\beta_I$, where $\beta=(\sigma_1 \sigma_2 \sigma_3\sigma_4 \sigma_5)^2$ and $I=\{2,3,5\}$.}
\label{F:subtrenza}
\end{figure}

We now go back to our original braid $\beta$ such that $CRS(\beta)\neq \emptyset$. Recall that we are trying to define a component $\beta_{\mathcal C}$ associated to $\mathcal C\in CRS(\beta)\cup\{\partial(D_n)\}$. Recall also the curves $\mathcal C_1,\ldots,\mathcal C_r$ defined above. We can now define a subset $I\subset \{1,\ldots,n\}$ containing all punctures enclosed by $\mathcal C$ and not enclosed by any ${\mathcal C}_i$, together with one puncture enclosed by $\mathcal C_i$, for $i=1,\ldots,r$. We could then define $\beta_{\mathcal C}$ as the subbraid $\beta_I$, as this fits with our intuitive idea. Unfortunately, this is not well defined, as the following example shows.

Consider, the braid $\beta=\sigma_2\sigma_2\sigma_1 \sigma_3\in B_4$, whose canonical reduction system $CRS(\beta)=\{\mathcal C_1,\mathcal C_2\}$ is drawn in Figure~\ref{F:trenza_2}. It is clear that $\beta_{\mathcal C_1}$ and $\beta_{\mathcal C_2}$ should be trivial braids on two strands. But if we try to take a subbraid of $\beta$ to define $\beta_{\partial(D_4)}$, this would depend on the choice of punctures: $\beta_{\{1,2\}}=\sigma_1\in B_2$, while $\beta_{\{1,4\}}=1\in B_2$. This problem comes from the fact that $\mathcal C_1$ and $\mathcal C_2$ are not round curves, otherwise the choice of the strands inside these curves would be irrelevant.

\begin{figure}[ht]
\centerline{\includegraphics{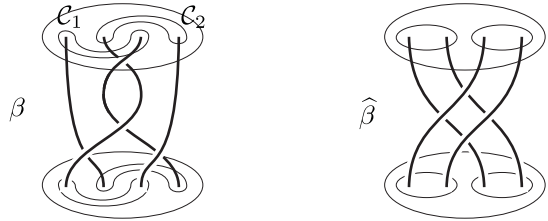}} \caption{The braid $\beta=\sigma_2 \sigma_2 \sigma_1\sigma_3$ and its conjugate by $\sigma_2$,  $\widehat\beta = \sigma_2 \sigma_1\sigma_3\sigma_2$.
}
\label{F:trenza_2}
\end{figure}

The solution to the above problem is given by Sang Jin Lee and Eon-Kyung Lee~\cite{LL}. In this remarkable paper they show that, given an isotopy class $[\mathcal F]$ of a family of curves in $D_n$, there is a {\it unique} positive braid $\alpha\in B_n$ such that $[\mathcal F]^{\alpha}$ is round, and $\alpha$ has minimal length among all positive braids satisfying this property. Actually, $\alpha$ is a prefix of any other positive braid transforming $\mathcal F$ into a family of round curves. Let us call $\alpha$ the {\it minimal standardizer} of $\mathcal F$.

Denote $\widehat \beta= \alpha^{-1}\beta\alpha$, and notice that $\alpha$ sends the curves $\mathcal C,\mathcal C_1,\ldots,\mathcal C_r$ to round curves that we denote, respectively,
$\widehat{\mathcal C},\widehat{\mathcal C_1},\ldots,\widehat{\mathcal C_r}$. We will finally define $\beta_{\mathcal C}$ as the subbraid of $\widehat\beta$ obtained by taking one strand inside each $\widehat{\mathcal C_i}$, plus the strands corresponding to the punctures of $D_{\widehat{\mathcal C}}$. More precisely:

\begin{definition}\label{D:component}
Let $\beta\in B_n$ and $\mathcal C\in CRS(\beta)\cup \{\partial(D_n)\}$. Let $I\subset \{1,\ldots,n\}$ be obtained from $\overline{D_{\widehat{\mathcal C}}}\backslash D_{\widehat{\mathcal C}}$ by replacing each curve with a puncture enclosed by that curve. Then we define $\beta_{\mathcal C}$, the component of $\beta$ associated to $\mathcal C$, as the subbraid $(\widehat\beta)_I$.
\end{definition}

In the example given in Figure~\ref{F:trenza_2}, in order to obtain $\beta_{\partial(D_n)}$ one first needs to conjugate $\beta=\sigma_2\sigma_2\sigma_1\sigma_3$ by its minimal standardizer $\alpha=\sigma_2$, to obtain $\widehat\beta= \sigma_2\sigma_1\sigma_3\sigma_2$, where we can clearly see that $\beta_{\partial(D_n)} = \widehat\beta_{\{1,3\}} = \widehat\beta_{\{1,4\}} = \widehat\beta_{\{2,3\}} = \widehat\beta_{\{2,4\}} = \sigma_1\in B_2$.

We remark that this notion of component of a braid, although not stated exactly in this way, is the same as the one given in~\cite{LL}.

Another important remark, to avoid confusion, is that if one decomposes $\beta$ along $CRS(\beta)$ into smaller braids, these braids are not necessarily pseudo-Anosov or periodic, as it happened in Thurston's decomposition theorem.  In order to apply Thurston's theorem to the case of braids, one needs the following:

\begin{definition}
Let $\beta\in B_n$ and let $\mathcal F=CRS(\beta)\cup \{\partial(D_n)\}$. Given $\mathcal C\in \mathcal F$, let $\mathcal C_{i}\in[\mathcal C]^{\beta^i}$ for $i\geq 0$, and let $m$ be the smallest positive integer such that $[\mathcal C_m]=[\mathcal C_0]=[\mathcal C]$. Then we define the {\bf interior braid} of $\beta$ associated to $\mathcal C$ to be
$$
 \beta^{\circ}_{\mathcal C}:=(\beta^m)_{\mathcal C}=\beta_{\mathcal C_0}\beta_{\mathcal C_1}\cdots \beta_{\mathcal C_{m-1}}.
$$
\end{definition}

Thanks to this notion of interior braid we can use Thurston's theorem in the case of braids.
More precisely, $\beta_{\mathcal C}$ is not necessarily periodic or pseudo-Anosov as $\mathcal C$ is not necessarily preserved by $\beta$, but if we take a suitable power of $\beta$ which fixes $\mathcal C$, its corresponding component ($\beta^{\circ}_{\mathcal C}$) is indeed either periodic or pseudo-Anosov.

We end this section by pointing out that in the definition of $\beta_{\mathcal C}$ we did not make use of the fact that $\mathcal C$ is an essential curve, but only that it belongs to a family of curves which is invariant under $\beta$. Actually, we can decompose a braid $\beta$ along any invariant family of curves $\mathcal F$ (containing $\partial(D_n)$), in the same way as above. The notation should be modified in this case, and we will denote $\beta_{[\mathcal C\in \mathcal F]}$ to be the component associated to $\mathcal C$ in this decomposition.  In particular $\beta_\mathcal C= \beta_{[\mathcal C\in CRS(\beta)]}$.

In some cases we do not even need the family $\mathcal F$ to be preserved by $\beta$. If $\mathcal F$ is a family of round curves such that $[\mathcal F]^\beta$ is also a family of round curves, we can just define $\beta_{[\mathcal C\in \mathcal F]}$ as in Definition~\ref{D:component} by replacing $CRS(\beta)$ with $\mathcal F$, $D_{\widehat{\mathcal C}}$ with $D_{\mathcal C}$ and $\widehat\beta$ with $\beta$. So in this case $\beta_{[\mathcal C\in \mathcal F]}$ is the subbraid $\beta_{I}$, where $I$ is obtained from $\overline{D_{\mathcal C}}\backslash D_{\mathcal C}$ by replacing each curve with a puncture enclosed by that curve. This subbraid is well defined thanks to the roundness of $[\mathcal F]$ and $[\mathcal F]^\beta$.

\section{Garside structure: Cyclic sliding and sliding circuits}\label{S:Garside}

We now briefly recall the notions introduced in~\cite{GG1} to solve the conjugacy problem in Garside groups (in particular in braid groups), and which replace the previous notions of cyclings, decyclings~\cite{EM} and ultra summit sets~\cite{Gebhardt}.

Recall that the braid group (and every Garside group) has a lattice structure, so every two elements $\alpha,\beta\in B_n$ admit an element $\alpha\wedge \beta$, called greatest common prefix, which is their meet with respect to the partial order $\preccurlyeq$.  Using this, one can define a normal form of the elements in $B_n$, called the {\bf left normal form}\cite{Deligne,Adyan,EM,Epstein}, which is a decomposition of any element $x\in B_n$ as $\Delta^p x_1\cdots x_r$, where $p$ is the maximal integer such that $\Delta^{-p}x$ is positive, each $x_i$ is nontrivial, and $x_i=(x_i\cdots x_r)\wedge \Delta$, for $i= 1,\ldots,r$. The {\bf infimum}, {\bf supremum} and {\bf canonical length} of $x$ are defined, respectively, $\inf(x)=p$, $\sup(x)=p+r$ and $\ell(x)=r$.

The factors $x_1,\ldots,x_r$ in the above decomposition are {\bf simple elements}, that is, positive prefixes of $\Delta$. If $s$ is simple, denote $\partial(s)=s^{-1}\Delta$, the {\bf complement} of $s$, which is also simple. It is well known that if $a$ and $b$ are two simple elements, the left normal form of the product $ab$ is equal to $(as)t$ (the first factor is $as$ and the second one is $t$), where $s=\partial(a)\wedge b$. Notice that in this case $as$ could be equal to $\Delta$ and $t$ could be trivial.

Let $\tau$ be the inner automorphism of $B_n$ associated to $\Delta$, that is, $\tau(\alpha)=\Delta^{-1}\alpha\Delta$ for any $\alpha\in B_n$. From the left normal form $x=\Delta^p x_1\cdots x_r$ we can define the {\bf initial factor} of $x\in B_n$ as $\iota(x)=(x\Delta^{-p})\wedge \Delta$. That is, $\iota(x)=(\tau^{-p}(x_1\cdots x_r))\wedge \Delta$, so if $r=0$ one has $\iota(x)=1$, while if $r>0$ one obtains $\iota(x)=\tau^{-p}(x_1)$.  We define the {\bf final factor} of $x$ as $\varphi(x)= (\Delta^{p+r-1}\wedge x)^{-1}x$, which means $\varphi(x)=x_r$ if $r>0$ and $\varphi(x)=\Delta$ if $r=0$. It is well known that $\iota(x^{-1})=\partial(\varphi(x))$.

The {\bf preferred prefix} of $x$ is defined by $\mathfrak p(x)=\iota(x)\wedge \iota(x^{-1})$, and the {\bf cyclic sliding} of $x$ is the conjugate of $x$ by its preferred prefix, that is, $\mathfrak s(x)=\mathfrak p(x)^{-1} x \mathfrak p(x)$.  To see why this is a natural definition, see~\cite{GG1}. For the moment, just notice that if $x$ has nonzero canonical length, $\mathfrak p(x) = \iota(x^{-1})\wedge \iota(x)= \partial(x_r) \wedge  \tau^{-p}(x_1)$, hence the first factor in the left normal form of $x_r\tau^{-p}(x_1)$ is precisely $x_r\mathfrak p(x)$.

We say that an element $y\in B_n$ is in a {\bf sliding circuit} if $\mathfrak s^m(y)=y$ for some $m>0$.  The {\bf set of sliding circuits} of a braid $x$, denoted $SC(x)$, is the set of conjugates of $x$ belonging to a sliding circuit. There is a simple algorithm to solve the conjugacy problem in $B_n$ (and in any Garside group), by using cyclic slidings and by computing sets of sliding circuits~\cite[Algorithm 0]{GG2}.

In~\cite{GG1} it is shown that application of cyclic sliding will never increase the canonical length of an element. Hence, as the set of conjugates of a given element with bounded canonical length is finite, it follows that any braid $x$ can be conjugated to a braid in $SC(x)$ by applying iterated cyclic slidings to $x$.

\section{Cyclic sliding  and round reduction curves}\label{S:CS and curves}

\subsection{cyclic sliding preserves roundness}

In this section we shall investigate the behavior of the elements having round reduction curves, under application of a cyclic sliding. The first result we need to recall is that if the roundness of a curve is preserved by a braid~$x$, then it is preserved by each factor in the left normal form of~$x$.

\begin{theorem}\label{T:BGN}{\rm \cite{BGN,LL}}
Let $x\in B_n$ be a positive braid whose left normal form is $x_1\cdots x_r$. If~$[\mathcal C]$ is a round curve such that $[\mathcal C]^x$ is also round, then $[\mathcal C]^{x_1\cdots x_i}$ is round for $i=1,\ldots,r$.
\end{theorem}

Since $\Delta^{\pm 1}$ preserves the roundness of every curve, the above result can be applied to every braid, not necessarily positive.  This is used in~\cite{BGN} to show that, if a braid preserves a round curve, its cycling and its decycling also preserve round curves. This immediately implies that for every reducible braid~$x$, there is some element in its super summit set $SSS(x)$ which preserves a round curve~\cite{BGN}. Clearly, one can replace $SSS(x)$ by $USS(x)$ in the previous statement. Even better, one can replace it by $SC(x)$, as we will now see, but the proof of this fact is slightly different: we need to show the following result, concerning invariant families of round curves.

\begin{proposition}\label{P:sliding_preserves_roundness}
Let $x\in B_n$, and let~$\mathcal F$ be a family round curves such that $[\mathcal F]^x=[\mathcal F]$. Then $[\mathcal F]^{\mathfrak p(x)}$ is also a family of round curves.   Hence, if~$x$ preserves a family of round curves, then so does~$\mathfrak s(x)$.
\end{proposition}

\begin{proof}
Let $\Delta^p x_1\cdots x_r$ be the left normal form of~$x$. We can assume~$r>0$. By Theorem~\ref{T:BGN} applied to each particular curve of~$\mathcal F$, one has that $[\mathcal F]^{\Delta^p x_1}$ is a family of round curves, and further application of $\Delta^{-p}$ yields the roundness of $[\mathcal F]^{\Delta^p x_1 \Delta^{-p}}= [\mathcal F]^{\tau^{-p}(x_1)}$. In the same way, Theorem~\ref{T:BGN} tells us that the curves of $[\mathcal F]^{\Delta^p x_1\cdots x_{r-1}}$ are round. Let~$\mathcal F_1$ be a family of curves such that $[\mathcal F_1]=[\mathcal F]^{\Delta^p x_1\cdots x_{r-1}}$, and let~$\mathcal F_2$ be such that $[\mathcal F_2]=[\mathcal F]^{\tau^{-p}(x_1)}$.

Notice that $[\mathcal F_1]^{x_r} =[\mathcal F]^{ (\Delta^p x_1\cdots x_{r-1}) x_r} = [\mathcal F]^{x} = [\mathcal F]$. Therefore we have $[\mathcal F_1]^{x_r \tau^{-p}(x_1)} = [\mathcal F_2]$, where~$\mathcal F_1$ and~$\mathcal F_2$ are families of round curves. Now recall that the first factor in the left normal form of $x_r \tau^{-p}(x_1)$ is $x_r\mathfrak p(x)$.  By Theorem~\ref{T:BGN} again, we obtain that the curves of $[\mathcal F_1]^{x_r \mathfrak p(x)}$ are round. But $[\mathcal F_1]^{x_r \mathfrak p(x)}= [\mathcal F]^{\mathfrak p(x)}$, hence $[\mathcal F]^{\mathfrak p(x)}$ is a family of round curves, as we wanted to show.
\end{proof}

As one can find an element in $SC(x)$ starting from a braid $x$, by iterated application of cyclic sliding, the following consequence of the above result is immediate:

\begin{corollary}\label{C:round curves in SC^m}
For every reducible, non periodic braid $x\in B_n$  there is some $y\in SC(x)$ such that $CRS(y)$ consists of round curves. Moreover, all elements in the sliding circuit of~$y$ satisfy the same property.
\end{corollary}

\begin{proof}
Given $x$, there is always a conjugate $z$ of $x$ whose canonical reduction system consists of round curves. Applying iterated cyclic sliding to $z$ preserves the roundness of the canonical reduction system, hence one will eventually obtain an element in $y\in SC(x)$ whose canonical reduction system is round.
\end{proof}

Notice that the above proof does not provide an algorithm to find~$y$, since we do not know a priori which is the braid $z$ conjugate to $x$. Nevertheless, since $SC(x)$ is a finite set, one can compute the whole $SC(x)$ and check for each element whether it preserves some family of round curves. In this way one can find a reduction curve for~$y$, and then for~$x$.  The computation of the whole set $SC(x)$, starting from $x$, is given in~\cite{GG1}, and the way to check whether a given element preserves a collection of round curves can be found in~\cite{BGN}.

\subsection{Preferred prefix and round reduction curves}

We know that cyclic sliding preserves the roundness of the canonical reduction system of a braid. That is, if $CRS(y)$ is made of round curves, the preferred prefix $\mathfrak p(y)$ preserves this roundness.  But we can say much more about the preferred prefix, as we will see in this section. We will show that if $[\mathcal F]$ is a family of round curves, and $[\mathcal F]^y=[\mathcal F]$, then  the components of $\mathfrak p(y)$ with respect to $\mathcal F$ are determined by the left normal forms of the components of $y$ with respect to $\mathcal F$, in a very precise way.

Notice that if $\mathcal F$ is a family of curves such that $[\mathcal F]$ and $[\mathcal F]^y$ are round, the braid~$y$ is completely determined by~$\mathcal F$ and by its components with respect to $\mathcal F$. Moreover, the left normal form of~$y$ determines the left normal form of its components, and vice-versa. We will restrict to the positive case, for simplicity. For every positive integer $k$, denote $\Delta_k$ the half twist in the braid group $B_k$.

\begin{lemma}\label{L:Lee_Lee_Lemma}{\rm (\cite[Lemma 3.4 (x)]{LL})}
Let $y\in B_n^+$, and let $\mathcal F$ be a family of curves such that $[\mathcal F]$ and $[\mathcal F]^y$ are round. Then for every $\mathcal C \in \mathcal F$ enclosing $k$ punctures, one has $ y_{[\mathcal C\in \mathcal F]}\wedge \Delta_k= (y\wedge \Delta_n)_{[\mathcal C\in \mathcal F]}$.
\end{lemma}

Notice that given a positive braid $y$, the braid $y\wedge \Delta$ is the first factor in its left normal form unless $\Delta\preccurlyeq y$, in which case $y\wedge \Delta= \Delta$. Some authors prefer to define the left normal form of a positive braid $y$ as $y_1\cdots y_r$, where $y_i=(y_i\cdots y_r)\wedge \Delta$ for $i=1,\ldots,r$. This allows some of the leftmost factors to be equal to $\Delta$, while in the definition given in Section~\ref{S:Garside} the powers of $\Delta$ are not considered as factors of the normal form. Actually, with this latter definition one can take $r$ to be greater than the supremum of $y$, just by allowing some of the rightmost factors to be trivial. With this definition $y\wedge \Delta$ is the first factor in the left normal form of $y$, $\inf(y)$ is the number of $\Delta$ factors in the left normal form, and the canonical length of $y$ is the number of proper (nontrivial and non-$\Delta$) factors in the left normal form. We will call this kind of decomposition the {\bf positive} left normal form of $y$, which is uniquely defined up to adding some trivial factors at the end. The following result is an immediate consequence of Lemma~\ref{L:Lee_Lee_Lemma}.

\begin{lemma}\label{L:lnf_round_curves}
Let~$y\in B_n^+$, and let~$\mathcal F$ be a family of curves such that $[\mathcal F]$ and $[\mathcal F]^y$ are round. Let $y_1\cdots y_r$ be the positive left normal form of $y$. Let $[\mathcal C]\in [\mathcal F]$. For $i=1,\ldots,r$, denote~$[\mathcal C_i]= [\mathcal C]^{y_1\cdots y_{i-1}}$ and $[\mathcal F_i]=[\mathcal F]^{y_1\cdots y_{i-1}}$. Then the positive left normal form of $y_{[\mathcal C\in \mathcal F]}$ is precisely
$$
{y_1}_{[\mathcal C_1\in \mathcal F_1]}\: {y_2}_{[\mathcal C_2\in \mathcal F_2]} \cdots {y_r}_{[\mathcal C_r\in \mathcal F_r]}.
$$
\end{lemma}

In the case in which $y\in B_n$ is not positive, we just need to make a slight modification of the above result.

\begin{lemma}\label{L:lnf_round_curves_non_positive}
Let~$y\in B_n$, and let~$\mathcal F$ be a family of curves such that $[\mathcal F]$ and $[\mathcal F]^y$ are round. Let $\Delta^p y_1\cdots y_r$ be the left normal form of $y$. Let $[\mathcal C]\in [\mathcal F]$ enclosing $k$ strands. For $i=1,\ldots,r$, denote~$[\mathcal C_i]= [\mathcal C]^{\Delta^p y_1\cdots y_{i-1}}$ and $[\mathcal F_i]=[\mathcal F]^{\Delta^p y_1\cdots y_{i-1}}$. Then $(\Delta_k)^{-p}y_{[\mathcal C\in \mathcal F]}$ is a positive braid, and
its positive left normal form is precisely
$$
{y_1}_{[\mathcal C_1\in \mathcal F_1]}\: {y_2}_{[\mathcal C_2\in \mathcal F_2]} \cdots {y_r}_{[\mathcal C_r\in \mathcal F_r]}.
$$
\end{lemma}

\begin{proof}
We have $[\mathcal C]^{\Delta^{p}}=[\mathcal C_1]\in [\mathcal F_1]$ or equivalently $[\mathcal C_1]^{\Delta^{-p}}=[\mathcal C]$, and the component of $\Delta^{-p}$ associated to $\mathcal C_1$ is $(\Delta^{-p})_{[\mathcal C_1\in \mathcal F_1]}=\Delta_k^{-p}$.

Now consider the braid $\Delta^{-p}y=y_1\cdots y_r$. It is a positive braid such that $[\mathcal F_1]^{y_1\cdots y_r}=[\mathcal F]^y$ is round, and its positive left normal form is precisely $y_1\cdots y_r$. Its component with respect to $\mathcal C_1$ will also be positive, and it is equal to  $(\Delta^{-p}y)_{[\mathcal C_1\in \mathcal F_1]} = (\Delta^{-p})_{[\mathcal C_1\in \mathcal F_1]}\: y_{[\mathcal C\in \mathcal F]} = (\Delta_k)^{-p}y_{[\mathcal C\in \mathcal F]}$, hence the first claim is shown.  Since the above component is precisely $(y_1\cdots y_r)_{[\mathcal C_1\in F_1]}$, its positive left normal form is given in Lemma~\ref{L:lnf_round_curves}, and this finishes the proof.
\end{proof}

\begin{corollary}\label{C:initial_and_final_factors}
With the above notations, one has:
\begin{enumerate}

 \item $\inf(y)\leq \inf(y_{[\mathcal C\in \mathcal F]})\ $ and $\ \sup(y)\geq \sup(y_{[\mathcal C\in \mathcal F]})$.

 \item If $\ \inf(y)< \inf(y_{[\mathcal C\in \mathcal F]})\ $ then $\ \iota(y)_{[\mathcal C\in \mathcal F]}=\Delta$.

 \item If $\ \inf(y)= \inf(y_{[\mathcal C\in \mathcal F]})\ $ then $\ \iota(y)_{[\mathcal C\in \mathcal F]}=\iota(y_{[\mathcal C\in \mathcal F]})$.

 \item If $\ \sup(y)> \sup(y_{[\mathcal C\in \mathcal F]})\ $ then $\ (\varphi(y))_{[\mathcal C_r\in \mathcal F_r]}=1$.

 \item If $\ \sup(y)= \sup(y_{[\mathcal C\in \mathcal F]})\ $ then $\ (\varphi(y))_{[\mathcal C_r\in \mathcal F_r]}=\varphi(y_{[\mathcal C\in \mathcal F]})$.

\end{enumerate}
\end{corollary}

\begin{proof}
The only nontrivial fact to notice is that $\iota(y)$ sends $\mathcal F$ to a family of round curves. Indeed, if $\iota(y)=1$ there is nothing to prove (in this case both $y$ and $y_{[\mathcal C\in \mathcal F]}$ are powers of $\Delta$). If $\iota(y)\neq 1$ and $\inf(y)=p$, Theorem~\ref{T:BGN} implies the roundness of $[\mathcal F]^{\Delta^p y_{1}} = [\mathcal F]^{\iota(y)\Delta^{p}}$. Applying $\Delta^{-p}$ to this family of curves we obtain again a round family, hence $[\mathcal F]^{\iota(y)}$ is round, as we wanted to show. This implies that the braid $\iota(y)_{[\mathcal C\in \mathcal F]}$ in conditions 2 and 3 is well defined. The rest is immediate from Lemma~\ref{L:lnf_round_curves_non_positive}.
\end{proof}

Once shown that restrictions to components preserve left normal forms, we can show that they also preserve greatest common divisors.

\begin{lemma}\label{L:gcd_round_curves}
Let $x,y\in B_n$, and let~$\mathcal F$ be a family of curves such that $[\mathcal F]$, $[\mathcal F]^x$ and $[\mathcal F]^y$ are round. Then $[\mathcal F]^{x\wedge y}$ is also round, and for every $\mathcal C\in \mathcal F$, one has $(x\wedge y)_{[\mathcal C \in \mathcal F]}= x_{[\mathcal C\in \mathcal F]}\wedge y_{[\mathcal C\in \mathcal F]}$.
\end{lemma}

\begin{proof}
The first claim is shown be Lee and Lee~\cite[Theorem 4.2]{LL}. Now notice that multiplying~$x$ and~$y$ by any power of~$\Delta^2$ we change neither the hypothesis nor the thesis of the result, hence we can assume that~$x$ and~$y$ are positive braids.

As $x\wedge y$ is a prefix of $x$ and $y$, every subbraid of $x\wedge y$ is a prefix of the corresponding subbraids of $x$ and $y$. Hence $(x\wedge y)_{[\mathcal C \in \mathcal F]}\preccurlyeq x_{[\mathcal C\in \mathcal F]}$ and $(x\wedge y)_{[\mathcal C\in \mathcal F]}\preccurlyeq y_{[\mathcal C\in \mathcal F]}$, so one has $(x\wedge y)_{[\mathcal C \in \mathcal F]}\preccurlyeq x_{[\mathcal C\in \mathcal F]}\wedge y_{[\mathcal C\in \mathcal F]}$.

On the other hand, let~$k$ be the number of punctures enclosed by~$\mathcal C$, and suppose that a braid~$z\in B_k$ is a prefix of $x_{[\mathcal C\in \mathcal F]}$. Then there is a (unique) braid~$\overline z$ on~$n$ strands which sends~$[\mathcal F]$ to a family of round curves, say $[\mathcal G]$, such that $\overline z_{[\mathcal C\in \mathcal F]}=z$ and all other components are trivial. We will now see that $\overline z\preccurlyeq x$, that is,  $(\overline z)^{-1}x$ is positive. Indeed, we can define a braid $x'\in B_n$ which sends $[\mathcal G]$ to a family of round curves, by defining its components with respect to $\mathcal G$. For every $\mathcal D\in \mathcal F$, we have $[\mathcal D]^{\overline z}=[\mathcal D']\in \mathcal G$. If $\mathcal D\neq \mathcal C$, we define $x'_{[\mathcal D'\in \mathcal G]}=x_{[\mathcal D\in \mathcal F]}$. Finally we define $x'_{[\mathcal C'\in \mathcal G]}=z^{-1}(x_{[\mathcal C\in \mathcal F]})$, which is a positive braid. Since $[\mathcal F]^{\overline z}=[\mathcal G]$ and $[\mathcal G]^{x'}$ is round, when we multiply $\overline z$ and $x'$ we are merely multiplying their components, hence $\overline z \: x'= x$. As all components of $x'$ are positive, $x'$ is positive, therefore $\overline z\preccurlyeq x$.

In the same way, if~$z$ is a prefix of $y_{[\mathcal C\in \mathcal F]}$, we have $\overline z \preccurlyeq y$. Hence, if $z\preccurlyeq x_{[\mathcal C\in \mathcal F]}\wedge y_{[\mathcal C\in \mathcal F]}$, one has $\overline z \preccurlyeq x\wedge y$, and then $z=\overline z_{[\mathcal C\in \mathcal F]} \preccurlyeq (x\wedge y)_{[\mathcal C \in \mathcal F]}$. Therefore, $x_{[\mathcal C\in \mathcal F]}\wedge y_{[\mathcal C\in \mathcal F]}\preccurlyeq (x\wedge y)_{[\mathcal C \in \mathcal F]}$, and the equality holds.
\end{proof}

We can finally describe the preferred prefix of a braid $y$ preserving a family of round curves $\mathcal F$. Notice that we cannot assume only that $y$ preserves the roundness of $\mathcal F$, as the preferred prefix of $y$ involves also the components of $y^{-1}$, and they will be related to the components of $y$ in an appropriate way only if they are considered with respect to the same family. This occurs when $\mathcal F$ is preserved by $y$.

\begin{proposition}\label{P:preferred_prefix_round_curves}
Let $y\in B_n$, and let~$\mathcal F$ be a family of round curves such that $[\mathcal F]^y=[\mathcal F]$. Let $\mathcal C, \mathcal C'\in \mathcal F\cup\{\partial (D)\}$ be such that $[\mathcal C']^y=[\mathcal C]$. Then one has:
\begin{enumerate}

 \item If $\inf(y_{[\mathcal C\in \mathcal F]})=\inf(y)$ and $\sup(y_{[\mathcal C'\in \mathcal F]})=\sup(y)$ then
 $$
 \mathfrak p(y)_{[\mathcal C\in \mathcal F]} = \iota(y_{[\mathcal C\in\mathcal F]})\wedge \iota((y^{-1})_{[\mathcal C\in \mathcal F]})
 = \iota(y_{[\mathcal C\in\mathcal F]})\wedge \iota((y_{[\mathcal C'\in \mathcal F]})^{-1}).
 $$

 \item If $\inf(y_{[\mathcal C\in \mathcal F]})>\inf(y)$ and $\sup(y_{[\mathcal C'\in \mathcal F]})=\sup(y)$ then
 $
 \mathfrak p(y)_{[\mathcal C\in \mathcal F]}= \iota((y_{[\mathcal C'\in \mathcal F]})^{-1}).
 $

 \item If $\inf(y_{[\mathcal C\in \mathcal F]})=\inf(y)$ and $\sup(y_{[\mathcal C'\in \mathcal F]})<\sup(y)$ then
 $
 \mathfrak p(y)_{[\mathcal C\in \mathcal F]}=\iota(y_{[\mathcal C\in\mathcal F]}).
 $

 \item If $\inf(y_{[\mathcal C\in \mathcal F]})>\inf(y)$ and $\sup(y_{[\mathcal C'\in \mathcal F]})<\sup(y)$ then
 $
 \mathfrak p(y)_{[\mathcal C\in \mathcal F]}=\Delta.
 $

\end{enumerate}
\end{proposition}

\begin{proof}
First notice that $(y^{-1})_{[\mathcal C\in \mathcal F]}= (y_{[\mathcal C'\in \mathcal F]})^{-1}$, so the second equality in Condition 1 holds. Now by Lemma~\ref{L:gcd_round_curves}, we know that
$$
\mathfrak p(y)_{[\mathcal C\in \mathcal F]}=(\iota(y)\wedge \iota(y^{-1}))_{[\mathcal C\in \mathcal F]} = \iota(y)_{[\mathcal C\in \mathcal F]}\wedge \iota(y^{-1})_{[\mathcal C\in \mathcal F]}.
$$
Corollary~\ref{C:initial_and_final_factors} tells us that $\iota(y)_{[\mathcal C\in \mathcal F]}$ is equal either to $\iota(y_{[\mathcal C\in \mathcal F]})$ or to $\Delta$, depending whether $\inf(y_{[\mathcal C\in \mathcal F]})$ is equal to $\inf(y)$ or not.  We will also see that $\iota(y^{-1})_{[\mathcal C\in \mathcal F]}$ is equal either to $\iota((y_{[\mathcal C'\in \mathcal F]})^{-1})$ or to $\Delta$, depending whether $\sup(y_{[\mathcal C'\in \mathcal F]})$ is equal to $\sup(y)$ or not, and the result will be shown.

Applying Corollary~\ref{C:initial_and_final_factors} to the curve $\mathcal C'$, it follows that
$\varphi(y)_{[\mathcal C'_r\in \mathcal F_r]}$ is equal either to $\varphi(y_{[\mathcal C'\in \mathcal F]})$ or to the trivial braid, depending whether $\sup(y_{[\mathcal C'\in \mathcal F]})$ is equal to $\sup(y)$ or not, where $\mathcal C'_r$ is a curve such that $[\mathcal C'_r]^{\varphi(y)}=[\mathcal C]$. Now just recall that $\iota(y^{-1})=\partial(\varphi(y))$, that is, $\varphi(y)\iota(y^{-1})=\Delta$. Hence
$$
\Delta_k = \Delta_{[\mathcal C'_r\in \mathcal F]} = (\varphi(y)\iota(y^{-1}))_{[\mathcal C'_r\in \mathcal F_r]} = \varphi(y)_{[\mathcal C'_r\in \mathcal F_r]}\iota(y^{-1})_{[\mathcal C\in \mathcal F]}.
$$
In other words, $\partial(\varphi(y)_{[\mathcal C'_r\in \mathcal F_r]}) = \iota(y^{-1})_{[\mathcal C\in \mathcal F]}$. This implies that if $\varphi(y)_{[\mathcal C'_r\in \mathcal F_r]}$ is trivial, one has $\iota(y^{-1})_{[\mathcal C\in \mathcal F]}=\Delta$, while if $\varphi(y)_{[\mathcal C'_r\in \mathcal F_r]}=\varphi(y_{[\mathcal C'\in \mathcal F]})$ one has $\iota(y^{-1})_{[\mathcal C\in \mathcal F]}=\partial(\varphi(y_{[\mathcal C'\in \mathcal F]}))= \iota((y_{[\mathcal C'\in \mathcal F]})^{-1})$, as we wanted to show.
\end{proof}

If $[\mathcal C]=[\mathcal C']$ in the above statement, the first condition means that restriction preserves preferred prefixes. In this case, we can just restate it in the following way:

\begin{proposition}\label{P:preferred_prefix_fixed_round_curves}
Let $y\in B_n$, and let~$\mathcal F$ be a family of round curves such that $[\mathcal F]^y=[\mathcal F]$. Let $\mathcal C\in \mathcal F$ such that $[\mathcal C]^y=[\mathcal C]$. Then one has:
\begin{enumerate}

 \item If $\inf(y_{[\mathcal C\in \mathcal F]})=\inf(y)$ and $\sup(y_{[\mathcal C\in \mathcal F]})=\sup(y)$ then
 $
 \mathfrak p(y)_{[\mathcal C\in \mathcal F]}=\mathfrak p(y_{[\mathcal C\in \mathcal F]}).
 $

 \item If $\inf(y_{[\mathcal C\in \mathcal F]})>\inf(y)$ and $\sup(y_{[\mathcal C\in \mathcal F]})=\sup(y)$ then
 $
 \mathfrak p(y)_{[\mathcal C\in \mathcal F]}= \iota(y^{-1}_{[\mathcal C\in \mathcal F]}).
 $

 \item If $\inf(y_{[\mathcal C\in \mathcal F]})=\inf(y)$ and $\sup(y_{[\mathcal C\in \mathcal F]})<\sup(y)$ then
 $
 \mathfrak p(y)_{[\mathcal C\in \mathcal F]}=\iota(y_{[\mathcal C\in\mathcal F]}).
 $

 \item If $\inf(y_{[\mathcal C\in \mathcal F]})>\inf(y)$ and $\sup(y_{[\mathcal C\in \mathcal F]})<\sup(y)$ then
 $
 \mathfrak p(y)_{[\mathcal C\in \mathcal F]}=\Delta.
 $

\end{enumerate}
\end{proposition}

Notice that in the above statement the term $y^{-1}_{[\mathcal C\in \mathcal F]}$ is unambiguous, since $(y^{-1})_{[\mathcal C\in \mathcal F]}=(y_{[\mathcal C\in \mathcal F]})^{-1}$ as $[\mathcal C]$ is preserved by $y$.

\section{Big sets of cyclic slidings}\label{S:big_sets}

Having identified the preferred prefix of a reducible braid (with round reduction curves) in terms of its components, we observe the following facts. Suppose that we apply a cyclic sliding to a non-periodic, reducible braid $y$. Then, if a component of $y$ has the same infimum and supremum as $y$, we have applied a {\it cyclic sliding} to this component. If it has the same infimum but distinct supremum, we have applied a {\it cycling} to this component. If it has the same supremum but distinct infimum, we have applied a cycling to its inverse, which corresponds to apply a {\it decycling} followed by $\tau$ (conjugation by $\Delta$). Finally, if both the infimum and supremum of the component are distinct from those of $y$, we are merely conjugating this component by $\Delta$ (applying $\tau$).

From these facts it follows that if $y$ is reducible, applying iterated cyclic sliding to $y$ does not necessarily simplify each of its components. It is then quite easy to give examples of sets of sliding circuits which are exponential with respect to the number of strands. But there are already examples of this kind; see for instance~\cite{GG1} or~\cite{Prasolov}. Nevertheless, using this technique we will be able to find families of examples whose set of sliding circuits is exponential both in the number of strands and in the length of the braid. This is the first kind of examples having the second feature, to our knowledge.

In order to construct our example, we recall from~\cite{BGG3} that the braid $\delta = \sigma_{1}\sigma_{2}\cdots \sigma_{k-1}\in B_k$ has exactly $2^{k-2}$ conjugates which are simple (infimum zero and canonical length one). The precise permutation of each of these simple elements is given in~\cite[Proposition 10]{BGG3}. It is immediate to deduce that these $2^{k-2}$ elements are precisely the braids of the form:
$$
  \delta_{d_1,\ldots,d_m}=
   (\sigma_{d_1-1} \sigma_{d_1-2}\cdots \sigma_1) (\sigma_{d_2-1}\sigma_{d_2-2}\cdots \sigma_{d_1})\cdots (\sigma_{n-1}\sigma_{n-2}\cdots \sigma_{d_{m}}),
$$
for any choice of integers $1<d_1<d_2<\cdots<d_m<n$. Since they are exactly in one to one correspondence with the subsets of $\{2,\ldots,n-1\}$, this is why they are precisely $2^{n-2}$. We will also define $1<u_1<u_2<\cdots<u_k<n$ such that $\{u_1,\ldots,u_k\}=\{2,\ldots,n-1\}\backslash \{d_1,\ldots,d_m\}$.

Given a positive braid $\alpha\in B_n$, we will denote $S(\alpha)=\{\sigma_i\ ; \ \sigma_i\preccurlyeq \alpha\}$ and $F(\alpha)=\{\sigma_i\ ; \ \alpha \succcurlyeq \sigma_i\}$, the starting and finishing sets of $\alpha$, respectively.

\begin{lemma}\label{L:SF} Given $\delta_{d_1,\ldots,d_m}$, denote $u_0=d_0=1$. One has:
\begin{enumerate}
\item $S(\delta_{d_1,\ldots,d_m})=\{\sigma_{u_i} \ ; \ u_{i}+1\neq u_{i+1} \}$.

\item $F(\delta_{d_1,\ldots,d_m})=\{\sigma_{d_i} \ ; \ d_{i}+1\neq d_{i+1} \}$.
\end{enumerate}
In particular, $S(\delta_{d_1,\ldots,d_m})\cap F(\delta_{d_1,\ldots,d_m})=\emptyset$.
\end{lemma}

\begin{proof}
From~\cite{BGG3}, the permutation associated to $\delta_{d_1,\ldots,d_m}$ is given by the single cycle $\pi=(1\ u_1\ u_2\cdots u_k \ n\ d_m\ d_{m-1}\cdots d_1)$. As $\delta_{d_1,\ldots,d_m}$ is a simple braid, we recall that its starting set is given by the generators $\sigma_i$ such that $\pi(i)>\pi(i+1)$. Suppose that $i=d_j$ for some $j$. Then $\pi(i)<i$. In this case, if $i+1=d_{j+1}$ or $i+1=n$ one has $\pi(i+1)=d_j=i>\pi(i)$, and if $i+1=u_k$ for some $k$, one has $\pi(i+1)>i+1>\pi(i)$. In either case $\sigma_i\neq S(\delta_{d_1,\ldots,d_m})$. Suppose now that $i=u_k$ for some $k$. Then $\pi(i)>i$. In this case, if $i+1=u_{k+1}$ one has $\pi(i)=i+1<\pi(i+1)$, so $\sigma_i\neq S(\delta_{d_1,\ldots,d_m})$. But if $i+1=d_j$ for some $j$ or $i+1=n$, one has $\pi(i)>i\geq \pi(i+1)$, so $\sigma_i\in S(\delta_{d_1,\ldots,d_m})$. The first claim is then shown.

The second claim follows from the first one, once we notice that for every positive braid $x$ one has $F(x)=S(\overleftarrow{x})$, where $\overleftarrow{x}$ is the braid obtained from any positive word representing $x$, read backwards (or the image of $x$ under the anti-isomorphism of $B_n$ which sends each $\sigma_i$ to itself). Since $\overleftarrow{\delta_{d_1,\ldots,d_m}}=\delta_{u_1,\ldots,u_k}$, the second claim follows.

Finally, the only possible element in $S(\delta_{d_1,\ldots,d_m})\cap F(\delta_{d_1,\ldots,d_m})$ is $\sigma_1$, but $\sigma_1$ belongs either to $S(\delta_{d_1,\ldots,d_m})$ or to $F(\delta_{d_1,\ldots,d_m})$ depending whether $2= d_1$ or $2=u_1$. Since both properties are mutually exclusive, the intersection is empty.
\end{proof}

We still need to define some other special elements for our example:

\begin{lemma}
For every $i,j\in \{1,\ldots,n-1\}$ there is a simple braid $\alpha_{i,j}$ such that $S(\alpha)=\{\sigma_i\}$ and $F(\alpha)=\{\sigma_j\}$.
\end{lemma}

\begin{proof}
It suffices to take, if $i\leq j$, $\alpha=\sigma_i\sigma_{i+1}\cdots \sigma_{j}$, and if $i\geq j$, $\alpha=\sigma_i\sigma_{i-1}\cdots \sigma_j$.
\end{proof}

\begin{lemma}\label{L:x}
Let $i_1,i_2,\ldots,i_{k+1}\in\{1,\ldots,n-1\}$, and let $\eta$ be a simple conjugate of $\delta$ such that $\sigma_{i_1}\in F(\eta)$. Then the element
$$
x_{\eta, i_1,\ldots,i_{k+1}}= (\alpha_{i_1,i_2}\alpha_{i_2,i_3}\cdots \alpha_{i_k,i_{k+1}})^{-1} \eta \;(\alpha_{i_1,i_2}\alpha_{i_2,i_3}\cdots \alpha_{i_k,i_{k+1}})
$$
is a conjugate of $\delta$ whose left normal form is:
$$
   \Delta^{-k}\; \partial^{-2k+1}(\alpha_{i_k,i_{k+1}})\cdots  \partial^{-3}(\alpha_{i_2,i_3}) \;\partial^{-1}(\alpha_{i_1,i_2})  \;\eta\;\alpha_{i_1,i_2}\;\alpha_{i_2,i_3}\cdots \alpha_{i_{k},i_{k+1}}.
$$
\end{lemma}

\begin{proof}
The first statement is evident as  $x_{\eta,i_1,\ldots,i_{k+1}}$ is a conjugate of $\eta$, which is a conjugate of $\delta$. Now, it is well known that a product of two simple factors $ab$ is left-weighted if and only if $S(b)\subset F(a)$. Since $F(\alpha_{i_{j-1},i_j})=\{\sigma_{i_j}\}=S(\alpha_{i_j,i_{j+1}})$, the factorization $\alpha_{i_{j-1},i_j}\alpha_{i_j,i_{j+1}}$ is left-weighted for $j=2,\ldots,k$. Also $S(\alpha_{i_1,i_2})=\{\sigma_{i_1}\}\subset F(\eta)$, hence the factorization $\eta\:\alpha_{i_1,i_2}$ is also left-weighted.

If we know the left normal form of a braid, the left normal form of its inverse is also known~\cite{EM}. In this case, since $\alpha_{i_1,i_2}\alpha_{i_2,i_3}\cdots \alpha_{i_k,i_{k+1}}$ is in left normal form as written, it follows that the left normal form of
$(\alpha_{i_1,i_2}\alpha_{i_2,i_3}\cdots \alpha_{i_k,i_{k+1}})^{-1}$  is equal to $  \Delta^{-k} \partial^{-2k+1}(\alpha_{i_k,i_{k+1}})\cdots  \partial^{-3}(\alpha_{i_2,i_3}) \partial^{-1}(\alpha_{i_1,i_2})$.

It only remains to show that $\partial^{-1}(\alpha_{i_1,i_2})\:\eta$ is left-weighted. But if $a$ and $b$ are two simple elements such that $ab=\Delta$, one has $F(a)\cup S(b)=\{1,\ldots,n-1\}$, that is $F(\partial^{-1}(b))=\{1,\ldots,n-1\}\backslash S(a)$. In this case $F(\partial^{-1}(\alpha_{i_1,i_2})) = \{1,\ldots,n-1\}\backslash S(\alpha_{i_1,i_2}) = \{1,\ldots,n-1\}\backslash\{\sigma_{i_1}\}$. On the other hand $\eta$ is a simple conjugate of $\delta$, so Lemma~\ref{L:SF} tells us that $S(\eta)\cap F(\eta)=\emptyset$. Since $\sigma_{i_1}\in F(\eta)$, it follows that $
\sigma_{i_1}\notin S(\eta)$, that is, $S(\eta)\subset \{1,\ldots,n-1\}\backslash \{\sigma_{i_1}\}= F(\partial^{-1}(\alpha_{i_1,i_2}))$, hence $\partial^{-1}(\alpha_{i_1,i_2})\eta$ is left-weighted, as we wanted to show.
\end{proof}

\begin{corollary}\label{C:x}
For every $i_1,\ldots,i_{2k}\in\{1,\ldots,n-1\}$ and every simple conjugate $\eta$ of $\delta$ such that $\sigma_{i_1}\in F(\eta)$, the braid $\Delta^{2k} x_{\eta,i_1,\ldots,i_{2k}}$ is a conjugate of $\delta^{nk+1}$ whose infimum is 1 and whose canonical length is $4k-1$. Moreover $x_{\eta,i_1,\ldots,i_{2k}}=x_{\gamma,j_1,\ldots,j_{2k}}$ if and only if  $\eta=\gamma$ and $(i_1,\ldots,i_{2k})=(j_1,\ldots,j_{2k})$.
\end{corollary}

\begin{proof}
We saw in the previous result that $x_{\eta,i_1,\ldots,i_{2k}}$ has infimum $-2k+1$ and canonical length $4k-1$. If we multiply this braid by $\Delta^{2k}$, we will obtain a braid whose infimum is 1 and whose canonical length is still $4k-1$ as stated. Since $x_{\eta,i_1,\ldots,i_{2k}}$ is a conjugate of $\delta$ and $\Delta^2=\delta^n$ is a central element of $B_n$, it follows that $\Delta^{2k} x_{\eta,i_1,\ldots,i_{2k}}$ is a conjugate of $\Delta^{2k}\delta=\delta^{nk+1}$.

Multiplying an element from the left by a power of $\Delta$ does not modify the non-$\Delta$ factors of its left normal form. This implies that the final $2k$ factors of the left normal form of $\Delta^{2k} x_{\eta,i_1,\ldots,i_{2k}}$ are precisely $\eta\alpha_{i_1,i_2}\alpha_{i_2,i_3}\cdots \alpha_{i_{2k-1},i_{2k}}$. As the element $\alpha_{i,j}$ is uniquely determined by the indices $i$ and $j$, it follows that the braid $\Delta^{2k} x_{\eta,i_1,\ldots,i_{2k}}$ is uniquely determined by $\eta$ and by the indices $(i_1,\ldots,i_{2k})$, as we wanted to show.
\end{proof}

We already have all the ingredients to provide a family of elements whose set of sliding circuits is exponential both on the number of strands and on the canonical length.

\begin{proposition}For $n\geq 3$ and $k\geq 1$, consider the braid $\beta\in B_{n+2}$ given by:
$$
\beta= (\sigma_1\cdots \sigma_{n-1})^{nk+1}\:(\sigma_{n+1})^{4k+1}.
$$
Then $\ell(\beta)=4k+1$ and $\#(SC(\beta))\geq 2^{n-2}(n-1)^{2k-1}$.
\end{proposition}

\begin{proof}
We remark that $SC(\beta)$ is much bigger than $2^{n-2}(n-1)^{2k-1}$, but we chose this bound for the simplicity of the proof, as it is already exponential on the braid index ($n+2$) and on the canonical length $(4k+1)$, as we shall see.

Denote $\delta=\sigma_1\cdots \sigma_{n-1}$, so $\beta=\delta^{nk+1}\:(\sigma_{n+1})^{4k+1}$. We recall that as $\delta^n=\Delta_n^2$ one has $\delta^{nk+1}=\Delta_n^{2k}\delta$, so the left normal form of $\beta$ is $(\Delta_n\sigma_{n+1})^{2k}(\delta \sigma_{n+1})(\sigma_{n+1})^{2k}$, which is the left-weighted product of $4k+1$ simple elements (the parenthesized terms correspond to simple factors). Hence $\inf(\beta)=0$ and $\ell(\beta)=4k+1$. Notice also that this braid is non-periodic and reducible: it is non-periodic as it is written as a word in which $\sigma_{n}$ does not appear, so the same must happen to every power of $\beta$, hence no power of $\beta$ can be equal to a power of $\Delta_{n+2}$. It is reducible since it preserves curve-wise the family of circles $\mathcal F=\{\mathcal C_{1,n},\mathcal C_{n+1,n+2}\}$. Actually $CRS(\beta)=\mathcal F$ since $\beta_{[\partial(D_n)\in \mathcal F]}$ is trivial, $\beta_{[\mathcal C_{1,n}\in \mathcal F]}= \delta^{nk+1}$ is periodic and nontrivial, and $\beta_{[\mathcal C_{n+1,n+2}\in F]}=\sigma_1^{4k+1}$ is also periodic and nontrivial, but we will not make use of this fact.

Now, for every simple conjugate $\eta$ of $\delta$ in $B_n$, we can choose an index $i_{\eta}$ such that $\sigma_{i_\eta}\in F(\eta)$. Then for every $i_2,\ldots,i_{2k}\in \{1,\ldots,n-1\}$, denote $x=(\Delta_n)^{2k} x_{\eta,i_\eta,i_2,\ldots,i_{2k}}$ and $y=x\:(\sigma_{n+1})^{4k+1}$. We saw in Corollary~\ref{C:x} that $x$ is a conjugate of $\delta^{nk+1}$ in $B_n$, so there exists $\gamma\in B_n$ such that $\gamma^{-1}x\gamma = \delta^{2k+1}$. We can consider $\gamma$ as a braid in $B_{n+2}$, by the standard embedding $B_n\rightarrow B_{n+2}$ that sends $\sigma_i\in B_n$ to $\sigma_i\in B_{n+2}$ (this corresponds to adding two vertical strands to the right of each braid). Then in $B_{n+2}$ one has $\gamma^{-1}y\gamma = \gamma^{-1}(x (\sigma_{n+1})^{4k+1})\gamma = \delta^{nk+1}(\sigma_{n+1})^{4k+1}=\beta$. Hence $y$ is conjugate to our original braid $\beta$.

We will now show that $y$ belongs to a sliding circuit. By Corollary~\ref{C:x}, $x\in B_n$ has infimum $1$ and canonical length $4k-1$, so its left normal form is $\Delta_n s_1\cdots s_{4k-1}$ for some simple elements $s_1,\ldots,s_{4k-1}$ (which are specified in Corollary~\ref{C:x}). The left normal form of $y=x\:(\sigma_{n+1})^{4k+1}\in B_{n+2}$ is then $(\Delta_n \sigma_{n+1})(s_1\sigma_{n+1})\cdots (s_{4k-1}\sigma_{n+1})(\sigma_{n+1})$. None of these simple factors is equal to $\Delta_{n+1}$, hence $\inf(y)=0$ and $\sup(y)=4k+1$. The preferred prefix of this element is then $\mathfrak p(x(\sigma_{n+1})^{4k+1}) = \partial(\sigma_{n+1}) \wedge  (\Delta_n\sigma_{n+1}) = \Delta_n$. Notice that this corresponds to the fourth condition in Proposition~\ref{P:preferred_prefix_fixed_round_curves}, applied to $y$ and $\mathcal C_{1,n}$.

Since $\mathfrak p(y)=\Delta_n$, applying a cyclic sliding to $y$ merely conjugates its component $x$ by $\Delta_n$, that is, $\mathfrak s(y)=\Delta_n^{-1} y \Delta_n = x' (\sigma_{n+1})^{4k+1}$, where $x'=\Delta_n^{-1}x \Delta_n$. If we define $i_r'=n-i_r$ for $r=2,\ldots,2k$, also $i_\eta'=n-i_\eta$ and $\eta'=\Delta_n^{-1} \eta \Delta_n$, we see that $\eta'$ is a simple conjugate of $\delta$ such that $\sigma_{i_\eta'}\in F(\eta')$, hence $x'= (\Delta_n)^{2k} x_{\eta',i_\eta',i_2'\ldots,i_{2k}'}$. We can then apply the above argument to $\mathfrak s(y)$ and $x'$, to conclude that $\mathfrak p(\mathfrak s(y))=\Delta_n$, hence $\mathfrak s^2(y) = \Delta_n^{-2} y \Delta_n^{2} = y$. Therefore $y$ belongs to a sliding circuit, so $y\in SC(\beta)$.

Finally notice that $y$ was defined by choosing $\eta$ and $i_2,\ldots,i_{2k}$ (the index $i_\eta$ is determined by $\eta$). Distinct choices of these data yield different values of $x=y_{[\mathcal C_{1,n}\in \mathcal F]}$, hence different values of $y$. Therefore there are at least as many elements in $SC(\beta)$ as possibilities we have for choosing these values. There are $2^{n-2}$ choices for $\eta$ and $n-1$ choices for each $i_r$, for $r=2,\ldots,2k$, hence there are at least $2^{n-2} (n-1)^{2k-1}$ elements in $SC(\beta)$.
\end{proof}

We conclude by noticing that the above example, as well as most examples of these kind, can be treated in a more intelligent way than just computing its set of sliding circuits, provided one needs to know whether a given element is conjugate to it. If one knows its essential reduction curves (as it is the case), one can apply cyclic sliding to each component, defining a smaller subset of $SC(\beta)$ containing the conjugates of $\beta$ in which {\it every component} belongs to a sliding circuit.  In order to do something like that, one needs an efficient algorithm to detect the canonical reduction system of a braid. There are some algorithms to do this: one is given by Bestvina and Handel~\cite{BH} using the theory of train tracks, and there is another one which will be soon available~\cite{GW2}, which is an improvement of the one given in~\cite{BGN} and~\cite{BGN2}, using Garside theory. None of these algorithms are shown to be polynomial with respect to the number of strands or the length of the braid, although both seem to be very fast in most cases. There are some examples for which the algorithm in~\cite{BH} is not efficient. We do not know of any such example for the one in~\cite{GW2}, and it is conjectured to be polynomial. We refer to~\cite{GW2} for more details.

%----------------------------------------------------------------------------

%|<------------------------------------------------------------------------>|

\end{document}